\documentclass{amsart}
\usepackage{amssymb}
\usepackage{amsmath}
\usepackage{amsfonts}

\newtheorem{theorem}{Theorem}
\theoremstyle{plain}

\newtheorem{definition}{Definition}
\newtheorem{example}{Example}

\newtheorem{lemma}{Lemma}

\newtheorem{remark}{Remark}

\numberwithin{equation}{section}
\input{tcilatex}

\begin{document}
\title{On the approximation by convolution type double singular integral
operators}
\author{Mine MENEKSE YILMAZ}
\address{Gaziantep University, Faculty of Arts and Science, Department of
Mathematics, Gaziantep, Turkey}
\email{menekse@gantep.edu.tr}
\author{Lakshmi Narayan Mishra}
\address[L.N.Mishra]{ $^{1}$Department of Mathematics, Mody University of
Science and Technology, Lakshmangarh, Sikar Road, Sikar, Rajasthan-332 311\\
$^{2}$L. 1627 Awadh Puri Colony Beniganj, Phase III, Opposite Industrial
Training Institute (I.T.I.), Ayodhya Main Road, Faizabad 224 001, Uttar
Pradesh, India}
\email{lakshminarayanmishra04@gmail.com}
\author{Gumrah UYSAL}
\address{Department of Computer Technologies, Division of Technology of
Information Security, Karabuk University, Karabuk 78050, Turkey}
\email{guysal@karabuk.edu.tr}

\begin{abstract}
In this paper, {\small w}e prove the pointwise convergence and the rate of
pointwise convergence for a family of singular integral operators in
two-dimensional setting in the following form:%
\begin{equation*}
L_{\lambda }\left( f;x,y\right) =\underset{D}{\diint }f\left( t,s\right)
K_{\lambda }\left( t-x,s-y\right) dsdt,\text{ }\left( x,y\right) \in D,
\end{equation*}%
where $D=\left \langle a,b\right \rangle \times \left \langle c,d\right \rangle $
is an arbitrary closed, semi-closed or open rectangle in $\mathbb{R}^{2}$
and $\lambda \in \Lambda ,$ $\Lambda $ is a set of non-negative numbers with
accumulation point $\lambda _{0}$. Also, we provide an example to support
these theoretical results. In contrast to previous works, the kernel
function $K_{\lambda }\left( t,s\right) $ does not have to be even, positive
or 2$\pi -$periodic.
\end{abstract}

\maketitle

\section{\textbf{Introduction}}

Taberski \cite{Taberski1} studied the pointwise convergence of integrable
functions and the approximation properties of derivatives of integrable
functions in $L_{1}\left( -\pi ,\pi \right) $ by a family of convolution
type singular integral operators depending on two parameters in the
following form:

\begin{equation}
L_{\lambda }\left( f;x\right) =\underset{-\pi }{\overset{\pi }{\int }}%
f\left( t\right) K_{\lambda }\left( t-x\right) dt,\text{ \ }x\in \left( -\pi
,\pi \right) ,\text{ \ }\lambda \in \Lambda ,  \tag{1.1}
\end{equation}%
where $K_{\lambda }\left( t\right) $ is the kernel satisfying suitable
assumptions and $\Lambda $ is a given set of non-negative numbers with
accumulation point $\lambda _{0}.$

After Taberski's study \cite{Taberski1}, Gadjiev \cite{Gadjiev} and
Rydzewska \cite{Riz1} gave some approximation theorems concerning pointwise
convergence and the order of pointwise convergence for operators of type
(1.1) at a generalized Lebesgue point and a $\mu -$generalized Lebesgue
point of $f\in L_{1}\left( -\pi ,\pi \right) $, respectively. Later on,
Bardaro \cite{Bardaro1} estimated the degree of pointwise convergence of
general type singular integrals at generalized Lebesgue points of the
functions $f\in L_{1}\left( 
\mathbb{R}
\right) $.

In \cite{Taberski2}, Taberski explored the pointwise convergence of
integrable functions in $L_{1}\left( Q\right) $ by a three-parameter family
of convolution type singular integral operators of the form:%
\begin{equation}
L_{\lambda }\left( f;x,y\right) =\underset{Q}{\iint }f\left( t,s\right)
K_{\lambda }\left( t-x,s-y\right) dsdt,\text{ \ }\left( x,y\right) \in Q, 
\tag{1.2}
\end{equation}%
where $Q=\left \langle -\pi ,\pi \right \rangle \times \left \langle -\pi ,\pi
\right \rangle $ is a closed, semi-closed or open rectangle. In this work,
the kernel function $K_{\lambda }$ was non-negative, even and $2\pi -$%
periodic with respect to both variables, seperately. Then, Rydzewska \cite%
{Riz2} extended this work by obtaining the rate of pointwise convergence at
a $\mu -$generalized $d-$point. Here, note that all other assumptions on the
indicated operators were same with \cite{Taberski2}. In particular, for
further studies on the convergence of double singular integral operators, we
address the reader to \cite{Siu1, Siu2, uysal1, Gadjiev2}.

In this study, we also investigated the pointwise convergence and the rate
of convergence of the operators similar to the studies above. In contrast to
previous works, the kernel function $K_{\lambda }\left( t,s\right) $ does
not have to be even, non-negative and 2$\pi -$periodic with respect to each
variable.

The main contribution of this paper is to investigate the rate of pointwise
convergence of the convolution type singular integral operators in the
following form:%
\begin{equation}
L_{\lambda }\left( f;x,y\right) =\underset{D}{\iint }f\left( t,s\right)
K_{\lambda }\left( t-x,s-y\right) dsdt,\  \  \left( x,y\right) \in D  \tag{1.3}
\end{equation}%
where $D=\left \langle a,b\right \rangle \times \left \langle c,d\right \rangle $
is an arbitrary closed, semi-closed or open bounded rectangle in $\mathbb{R}%
^{2},$ to $f\in L_{1}\left( D\right) ,$ at a $\mu $-generalized Lebesgue
point as $\left( x,y,\lambda \right) \rightarrow \left( x_{0},y_{0},\lambda
_{0}\right) .$ Here, $L_{1}\left( D\right) $ is the collection of all
measurable functions $f$ for which $\left \vert f\right \vert $ is integrable
on $D,$ $\Lambda \subset $ $\overset{-}{%
\mathbb{R}
}$ is a set of non-negative indices with accumulation point $\lambda _{0}$.

The paper is organized as follows: In Section 2, we introduce the
fundamental definitions. In Section 3, we prove the existence of the
operators of type (1.3). In Section 4, we present two theorems concerning
the pointwise convergence of $L_{\lambda }\left( f;x,y\right) $ to $f\left(
x_{0},y_{0}\right) $ whenever $\left( x_{0},y_{0}\right) $ is a $\mu $%
-generalized Lebesgue point of $f$ , for the cases $D$ is bounded rectangle
and $D=%
\mathbb{R}
^{2}$. In Section 5, we establih the rate of convergence of operators\ of
type (1.3) to $f\left( x_{0},y_{0}\right) $ as $\left( x,y,\lambda \right) $
tends to $\left( x_{0},y_{0},\lambda _{0}\right) $ and we conclude the paper
with an example to support our results.

\section{Preliminaries}

In this section we introduce the main definitions used in this paper.

The following definition is obtained by modifying the $\mu -$generalized
Lebesgue point definition given in \cite{Riz2}.

\begin{definition}
A point $\left( x_{0},y_{0}\right) \in D$\ is called a $\mu -$generalized
Lebesgue point of function $f\in L_{1}\left( D\right) $\textit{\ if}%
\begin{equation*}
\underset{\left( h,k\right) \rightarrow \left( 0,0\right) }{\lim }\frac{1}{%
\mu _{1}(h)\mu _{2}(k)}\overset{h}{\underset{0}{\dint }}\overset{k}{\underset%
{0}{\int }}\left \vert f\left( t+x_{0},s+y_{0}\right) -f\left(
x_{0},y_{0}\right) \right \vert dsdt=0,
\end{equation*}%
where $\mu _{1}(h)=\underset{0}{\overset{h}{\tint }}\rho _{1}(t)dt>0,$ $%
0<h<\delta _{0}<\min \{b-a,d-c\}$ and $\rho _{1}(t)$ is an integrable and
non-negative function on $\left[ 0,\delta _{0}\right] $ and similarly, $\mu
_{2}(k)=\underset{0}{\overset{k}{\tint }}\rho _{2}(s)ds>0,$ $0<k<\delta
_{0}<\min \{b-a,d-c\}$ and $\rho _{2}(s)$ is an integrable and non-negative
function on $\left[ 0,\delta _{0}\right] .$
\end{definition}

Now, we define the new set of kernel functions by using some of the kernel
conditions presented in \cite{Siu1}.

\begin{definition}
$\left( Class\text{ }A\right) $\textbf{\ }We will say that the function $%
K_{\lambda }\left( t,s\right) $ belongs to class $A$, if the following
criterions are fulfilled:
\end{definition}

\begin{description}
\item[a] $K_{\lambda }\left( t,s\right) $ is defined and integrable as a
function of $(t,s)$ on $%
\mathbb{R}
^{2}$ for each fixed $\lambda \in \Lambda $ and $\left \Vert \left \vert
K_{\lambda }\right \vert \right \Vert _{L_{1}(%
\mathbb{R}
^{2})}\leq M,$ $\forall \lambda \in \Lambda .$

\item[b] For fixed $(t_{0},s_{0})\in D,$ $K_{\lambda }\left(
t_{0},s_{0}\right) $ tends to infinity as\ $\lambda $ tends to $\lambda
_{0}. $

\item[c] $\underset{\left( x,y,\lambda \right) \rightarrow \left(
x_{0},y_{0},\lambda _{0}\right) }{\lim }\left \vert \underset{\mathbb{R}^{2}}%
{\diint }K_{\lambda }\left( t-x,s-y\right) dsdt-1\right \vert =0.$

\item[d] $\underset{\lambda \rightarrow \lambda _{0}}{\lim }\left \vert
K_{\lambda }\left( \gamma ,0\right) \right \vert =\underset{\lambda
\rightarrow \lambda _{0}}{\lim }\left \vert K_{\lambda }\left( 0,\gamma
\right) \right \vert =0,\  \  \forall \gamma >0.$

\item[e] $\underset{\lambda \rightarrow \lambda _{0}}{\lim }\underset{%
\mathbb{R}
^{2}\backslash \left \langle -\gamma ,\gamma \right \rangle \times \left
\langle -\gamma ,\gamma \right \rangle }{\diint }\left \vert K_{\lambda
}\left( t,s\right) \right \vert dsdt=0,\  \  \forall \gamma >0.$

\item[f] There exist the numbers $\delta _{1}>0$ and $\delta _{2}>0$ such
that $\left \vert K_{\lambda }\left( t,s\right) \right \vert $ is
monotonically increasing with respect to $t$\ on $(-\delta _{1},0]$ and
monotonically decreasing on $[0,\delta _{1})$ and similarly \ $\left \vert
K_{\lambda }\left( t,s\right) \right \vert $is monotonically increasing with
respect to $s$\ on $(-\delta _{2},0]$ and monotonically decreasing on $%
[0,\delta _{2})$ for any $\lambda \in \Lambda $. Similarly, $\left \vert
K_{\lambda }\left( t,s\right) \right \vert $ is bimonotonically increasing
with respect to $(t,s)$\ on $[0,\delta _{1})\times \lbrack 0,\delta _{2})$
and $(-\delta _{1},0]\times (-\delta _{2},0]$ and bimonotonically decreasing
with respect to $(t,s)$ on $[0,\delta _{1})\times (\delta _{2},0]$ and $%
(-\delta _{1},0]\times \lbrack 0,\delta _{2})$ for any $\lambda \in \Lambda $%
.
\end{description}

\begin{remark}
If the function $g:%
\mathbb{R}
^{2}\rightarrow 
\mathbb{R}
$ is bimonotonic on $[\alpha _{1},\alpha _{2}]\times \lbrack \beta
_{1},\beta _{2}]\subset 
\mathbb{R}
^{2},$ then the equality given by 
\begin{eqnarray*}
V(g;[\alpha _{1},\alpha _{2}]\times \lbrack \beta _{1},\beta _{2}]) &=&%
\overset{\alpha _{2}}{\underset{\alpha _{1}}{\dbigvee }}\overset{\beta _{2}}{%
\underset{\beta _{1}}{\dbigvee }}\left( g(t,s)\right) \\
&=&\left \vert g(\alpha _{1},\beta _{1})-g(\alpha _{1},\beta _{2})-g(\alpha
_{2},\beta _{1})+g(\alpha _{2},\beta _{2})\right \vert
\end{eqnarray*}%
holds \cite{Taberski2, Ghorpade}.
\end{remark}

Throughout this paper, we suppose that the kernel $K_{\lambda }\left(
t,s\right) $\textit{\ }belongs to class\textit{\ }$A$ and $\delta _{0}\leq
\min \left \{ \delta _{1},\delta _{2}\right \} .$

\section{\textbf{Existence of the operators}}

\begin{lemma}
If $f\in L_{1}(D)$, then the operator $L_{\lambda }\left( f;x,y\right) $
defines a continuous transformation acting on $L_{1}(D\mathbb{)}$.
\end{lemma}

\begin{proof}
Since $L_{\lambda }(f;x,y)$ is linear, it is sufficient to show that the
expression given by%
\begin{equation*}
\left \Vert L_{\lambda }\right \Vert _{1}=\underset{f\neq 0}{\sup }\frac{%
\left \Vert L_{\lambda }(f;x,y)\right \Vert _{L_{1}(D)}}{\left \Vert f\right
\Vert _{L_{1}(D)}}
\end{equation*}%
is bounded.

Let $D=\left \langle a,b\right \rangle \times \left \langle c,d\right
\rangle $ be an arbitrary closed, semi-closed or open bounded rectangle in $%
\mathbb{R}^{2}.$ The function $f$ is extended to $%
\mathbb{R}
^{2}$ by defining $g$ such that

\begin{equation*}
g(t,s)=\left \{ 
\begin{array}{ccc}
f(t,s), & if & (t,s)\in D, \\ 
0, & if & (t,s)\in 
\mathbb{R}
^{2}\backslash D.%
\end{array}%
\right.
\end{equation*}%
Now, using Fubini's Theorem \cite{Butzer}, we have%
\begin{eqnarray*}
\left \Vert L_{\lambda }(f;x,y)\right \Vert _{L_{1}(D)} &=&\underset{D}{%
\diint }\left \vert \underset{%
\mathbb{R}
^{2}}{\diint }g(t,s)K_{\lambda }\left( t-x,s-y\right) dsdt\right \vert dydx
\\
&\leq &\underset{D}{\diint }\left( \underset{%
\mathbb{R}
^{2}}{\diint }\left \vert g(t+x,s+y)\right \vert \left \vert K_{\lambda
}\left( t-x,s-y\right) \right \vert dsdt\right) dydx \\
&=&\underset{%
\mathbb{R}
^{2}}{\diint }\left \vert K_{\lambda }\left( t-x,s-y\right) \right \vert
\left( \underset{D}{\diint }\left \vert g(t+x,s+y)\right \vert dxdy\right)
dydx \\
&=&\underset{%
\mathbb{R}
^{2}}{\diint }\left \vert K_{\lambda }\left( t-x,s-y\right) \right \vert
\left( \underset{%
\mathbb{R}
^{2}}{\diint }\left \vert g(t+x,s+y)\right \vert dxdy\right) dydx \\
&\leq &M\left \Vert f\right \Vert _{L_{1}(D)}.
\end{eqnarray*}%
One may prove the assertion for the case $D=%
\mathbb{R}
^{2}$ using similar method. The proof is completed.
\end{proof}

\section{\textbf{Pointwise} \textbf{convergence}}

The following theorem gives a pointwise approximation of the integral
operators of type (1.3) to the function $f$ at $\mu $-generalized Lebesgue
point of $f\in L_{1}(D)$ whenever $D$ is an arbitrary rectangle in $\mathbb{R%
}^{2}$ such that bounded, closed, semi-closed or open$.$

\begin{theorem}
If $\left( x_{0},y_{0}\right) $\ is a $\mu $-generalized Lebesgue point of
function $f\in L_{1}\left( D\right) ,$ then\textit{\ }%
\begin{equation*}
\underset{\left( x,y,\lambda \right) \rightarrow \left( x_{0},y_{0},\lambda
_{0}\right) }{\lim }L_{\lambda }\left( f;x,y\right) =f\left(
x_{0},y_{0}\right) 
\end{equation*}%
\textit{\ } on any set $Z$\ on which the functions%
\begin{equation}
\overset{x_{0}+\delta }{\underset{x_{0}-\delta }{\int }}\overset{%
y_{0}+\delta }{\underset{y_{0}-\delta }{\int }}\left \vert K_{\lambda
}(t-x,s-y)\right \vert \rho _{1}(\left \vert t-x_{0}\right \vert )\rho
_{2}(\left \vert s-y_{0}\right \vert )dsdt,  \tag{4.1}
\end{equation}%
and%
\begin{equation}
\left \vert K_{\lambda }\left( 0,0\right) \right \vert \mu _{1}(\left \vert
x-x_{0}\right \vert )\text{ \ and \ }\left \vert K_{\lambda }\left( 0,0\right)
\right \vert \mu _{2}(\left \vert y-y_{0}\right \vert )  \tag{4.2}
\end{equation}%
are bounded as $\left( x,y,\lambda \right) $\ tends to\textit{\ }$\left(
x_{0},y_{0},\lambda _{0}\right) $\textit{.} Here, the set $Z$ consists of
the points $(x,y,\lambda )\in D\times \Lambda $ such that the functions
given by (4.1) and (4.2) are bounded for all $\delta >0$ which satisfy $%
0<\delta <\delta _{0}.$
\end{theorem}

\begin{proof}
The proof begins with the following compulsory assumptions:\ Suppose that $%
\left( x_{0},y_{0}\right) \in D,$ $0<x_{0}-x<\delta /2,$ for all $\delta >0$
which satisfies $x_{0}+\delta <b$ and $x_{0}-\delta >a$, $0<y_{0}-y<\delta
/2 $ for all $\delta >0$ which satisfies $y_{0}+\delta <d$ and $y_{0}-\delta
>c$ whenever $0<\delta <\delta _{0}.$ Let us divide $D$ into the sets $%
D_{ij}=\left \langle c_{i},c_{i+1}\right \rangle \times \left \langle
d_{j},d_{j+1}\right \rangle $ for $1\leq i\leq 3$ and $1\leq j\leq 3,$ where 
$c_{1}=a,$ $c_{2}=x_{0}-\delta ,$ $c_{3}=x_{0}+\delta ,$ $c_{4}=b$ and $%
d_{1}=c,$ $d_{2}=y_{0}-\delta ,$ $d_{3}=y_{0}+\delta ,$ $d_{4}=d.$ Recall
the definition of the function $g(t,s)$ such that%
\begin{equation*}
g\left( t,s\right) =\left \{ 
\begin{array}{c}
f\left( t,s\right) ,\text{ \  \  \  \  \  \ }\left( t,s\right) \in D\text{,} \\ 
0,\text{ \  \  \  \  \  \ }\left( t,s\right) \in 
\mathbb{R}
^{2}\backslash D.%
\end{array}%
\right.
\end{equation*}%
Let $\left( x_{0},y_{0}\right) \in D$ is a $\mu $-generalized Lebesgue point
of function $f\in L_{1}\left( D\right) .$ Therefore, for all given $%
\varepsilon >0$, there exists a $\delta >0$ such that for all $h,k$
satisfying $0<h,k\leq \delta ,$ the inequality:%
\begin{equation}
\overset{x_{0}}{\underset{x_{0}}{\int }}\overset{+\delta y_{0}}{\underset{%
y_{0}-\delta }{\int }}\left \vert f\left( t,s\right) -f\left(
x_{0},y_{0}\right) \right \vert dsdt<\varepsilon \mu _{1}(h)\mu _{2}(k) 
\tag{4.3}
\end{equation}%
holds.

Set $I_{\lambda }(x,y):=\left \vert L_{\lambda }\left( f;x,y\right) -f\left(
x_{0},y_{0}\right) \right \vert $. According to condition (c) of class $A$,
we shall write%
\begin{eqnarray*}
I_{\lambda }(x,y) &=&\left \vert \underset{D}{\diint }f\left( t,s\right)
K_{\lambda }\left( t-x,s-y\right) dsdt-f\left( x_{0},y_{0}\right) \right
\vert \\
&\leq &\underset{\mathbb{R}^{2}}{\diint }\left \vert g\left( t,s\right)
-f\left( x_{0},y_{0}\right) \right \vert \left \vert K_{\lambda }\left(
t-x,s-y\right) \right \vert dsdt \\
&&+\left \vert f\left( x_{0},y_{0}\right) \right \vert \left \vert \underset{%
\mathbb{R}^{2}}{\diint }K_{\lambda }\left( t-x,s-y\right) dsdt-1\right \vert
\end{eqnarray*}%
\begin{eqnarray*}
&=&\underset{D}{\diint }\left \vert f\left( t,s\right) -f\left(
x_{0},y_{0}\right) \right \vert \left \vert K_{\lambda }\left(
t-x,s-y\right) \right \vert dsdt \\
&&+\left \vert f\left( x_{0},y_{0}\right) \right \vert \left \vert \underset{%
\mathbb{R}^{2}}{\diint }K_{\lambda }\left( t-x,s-y\right) dsdt-1\right \vert
\\
&&+\left \vert f\left( x_{0},y_{0}\right) \right \vert \underset{\mathbb{R}%
^{2}\backslash D}{\diint }\left \vert K_{\lambda }\left( t-x,s-y\right)
\right \vert dsdt \\
&=&I_{1}+I_{2}+I_{3}.
\end{eqnarray*}%
It is easy to see that $I_{2}\rightarrow 0$ as $\lambda \rightarrow \lambda
_{0}$ by condition (c) of class $A.$ On the other hand, since%
\begin{eqnarray*}
I_{3} &\leq &\left \vert f\left( x_{0},y_{0}\right) \right \vert \underset{%
\mathbb{R}^{2}\backslash \left \langle x-\delta ,x+\delta \right \rangle
\times \left \langle x-\delta ,x+\delta \right \rangle }{\diint }\left \vert
K_{\lambda }\left( t-x,s-y\right) \right \vert dsdt \\
&=&\left \vert f\left( x_{0},y_{0}\right) \right \vert \underset{%
\mathbb{R}
^{2}\backslash \left \langle -\delta ,+\delta \right \rangle \times \left
\langle -\delta ,+\delta \right \rangle }{\diint }\left \vert K_{\lambda
}\left( t,s\right) \right \vert dsdt,
\end{eqnarray*}%
by condition (e) of class $A$ we have $I_{3}\rightarrow 0$ as $\left(
x,y,\lambda \right) $\ tends to\textit{\ }$\left( x_{0},y_{0},\lambda
_{0}\right) .$

Obviously, the integral $I_{1}$ can be written as follows: 
\begin{eqnarray*}
&&I_{1}=\left \{ \underset{D\backslash D_{22}}{\diint }+\underset{D_{22}}{%
\diint }\right \} \left \vert f\left( t,s\right) -f\left( x_{0},y_{0}\right)
\right \vert \left \vert K_{\lambda }\left( t-x,s-y\right) \right \vert dsdt
\\
&=&I_{11}+I_{12}.
\end{eqnarray*}

We will show that $I_{11}\rightarrow 0$ as $\left( x,y,\lambda \right) $\
tends to\textit{\ }$\left( x_{0},y_{0},\lambda _{0}\right) $\textit{\ }by
following the similar steps as in \cite{Siu1}.\textit{\ }

Now, recall that the rectangle $D$ was splitted into nine rectangles at the
beginning of the proof. Therefore, $I_{11}$ consists of eight integrals.
Now, denote the integral which corresponds to $D_{12}$ by $I_{D_{12}}.$
Hence, by condition (f) of class $A,$ the following inequality 
\begin{eqnarray*}
~I_{D_{12}} &=&\overset{x_{0}}{\underset{a}{\int }}\overset{-\delta
y_{0}+\delta }{\underset{y_{0}-\delta }{\int }}\left \vert f\left( t,s\right)
-f\left( x_{0},y_{0}\right) \right \vert \left \vert K_{\lambda }\left(
t-x,s-y\right) \right \vert dsdt \\
&\leq &\left \vert K_{\lambda }\left( \delta /2,0\right) \right \vert \left(
\left \Vert f\right \Vert _{L_{1}\left( D\right) }+\left \vert f\left(
x_{0},y_{0}\right) \right \vert \left \vert b-a\right \vert \left \vert
d-c\right \vert \right) 
\end{eqnarray*}

holds. By condition (d) of class $A$, $I_{D_{12}}\rightarrow 0$ as $\left(
x,y,\lambda \right) $\ tends to\textit{\ }$\left( x_{0},y_{0},\lambda
_{0}\right) .$ The remaining seven integrals are evaluated by an analogous
method and this part is omitted.

It follows that%
\begin{equation*}
I_{11}\leq \underset{(t,s)\in D\backslash D_{22}}{\sup }\left \vert
K_{\lambda }\left( t-x,s-y\right) \right \vert \left( \left \Vert f\right
\Vert _{L_{1}\left( D\right) }+\left \vert f\left( x_{0},y_{0}\right) \right
\vert \left \vert b-a\right \vert \left \vert d-c\right \vert \right) .
\end{equation*}%
Consequently, $I_{11}\rightarrow 0$ as $\left( x,y,\lambda \right) $\ tends
to\textit{\ }$\left( x_{0},y_{0},\lambda _{0}\right) .$

Next, we prove that $I_{12}$ tends to zero as $\left( x,y,\lambda \right) $
tends to $\left( x_{0},y_{0},\lambda _{0}\right) $. Therefore, it is easy to
see that the following inequality holds for $I_{12}$, that is%
\begin{eqnarray*}
I_{12} &=&\left \{ \overset{x_{0}+\delta }{\underset{x_{0}}{\int }}\overset{%
y_{0}}{\underset{y_{0}-\delta }{\int }}+\overset{x_{0}}{\underset{%
x_{0}-\delta }{\int }}\overset{y_{0}}{\underset{y_{0}-\delta }{\int }}\right
\} \left \vert f\left( t,s\right) -f\left( x_{0},y_{0}\right) \right \vert
\left \vert K_{\lambda }(t-x,s-y)\right \vert dsdt \\
&&+\left \{ \overset{x_{0}}{\underset{x_{0}-\delta }{\int }}\overset{%
y_{0}+\delta }{\underset{y_{0}}{\int }}+\overset{x_{0}+\delta }{\underset{%
x_{0}}{\int }}\overset{y_{0}+\delta }{\underset{y_{0}}{\int }}\right \}
\left \vert f\left( t,s\right) -f\left( x_{0},y_{0}\right) \right \vert
\left \vert K_{\lambda }(t-x,s-y)\right \vert dsdt \\
&=&I_{121}+I_{122}+I_{123}+I_{124}.
\end{eqnarray*}%
Let us consider the integral $I_{121}.$ For the evaluations, we need to
define the following variations:%
\begin{eqnarray*}
A_{1}\left( u,v\right) &:&=\left \{ 
\begin{array}{c}
\overset{x_{0}+\delta -x}{\underset{u}{\dbigvee }}\overset{v}{\underset{%
y_{0}-\delta -y}{\dbigvee }}\left \vert K_{\lambda }\left( t,s\right) \right
\vert ,\text{ \ }x_{0}-x\leq u<x_{0}+\delta -x, \\ 
\text{ \  \  \  \  \  \  \  \  \  \  \  \  \  \  \  \  \  \  \  \  \  \  \  \  \  \  \  \  \  \  \  \  \  \  \
\  \  \  \  \  \  \  \  \  \ }y_{0}-\delta -y<v\leq y_{0}-y, \\ 
0,\text{ \  \  \  \  \  \  \  \  \  \  \  \  \  \  \  \  \  \  \  \  \  \  \  \  \  \  \  \  \  \  \
otherwise.}%
\end{array}%
\right. \\
A_{2}\left( u\right) &:&=\left \{ 
\begin{array}{c}
\overset{x_{0}+\delta -x}{\underset{u}{\dbigvee }}\left \vert K_{\lambda
}(t,y_{0}-\delta -y)\right \vert ,\text{ \  \ }x_{0}-x\leq u<x_{0}+\delta -x,
\\ 
0,\text{ \  \  \  \  \  \  \  \  \  \  \  \  \  \  \  \  \  \  \  \  \  \  \  \  \  \  \  \  \  \  \
otherwise.}%
\end{array}%
\right. \\
A_{3}\left( v\right) &:&=\left \{ 
\begin{array}{c}
\overset{v}{\underset{y_{0}-\delta -y}{\dbigvee }}\left \vert K_{\lambda
}(x_{0}-x+\delta ,y)\right \vert ,\text{ \ }y_{0}-\delta -y<v\leq y_{0}-y,
\\ 
0,\text{ \  \  \  \  \  \  \  \  \  \  \  \  \  \  \  \  \  \  \  \  \  \  \  \  \  \  \  \  \  \  \
otherwise.}%
\end{array}%
\right. .
\end{eqnarray*}%
Taking above variations and (4.3) into account and applying bivariate
integration by parts method (see, e.g., \cite{Taberski2}) to last
inequality, we have%
\begin{eqnarray*}
I_{121} &\leq &-\varepsilon \overset{x_{0}-x+\delta }{\underset{x_{0}-x}{%
\int }}\overset{y_{0}-y}{\underset{y_{0}-y-\delta }{\int }}\left[
A_{1}\left( t,s\right) +A_{2}\left( t\right) +A_{3}\left( s\right) +\left
\vert K_{\lambda }\left( x_{0}-x+\delta ,y_{0}-\delta -y\right) \right \vert %
\right] \\
&&\times \left \{ \mu _{1}\left( t-x_{0}+x\right) \right \} _{t}^{\prime
}\left \{ \mu _{2}\left( y_{0}-s-y\right) \right \} _{s}^{\prime }dsdt \\
&=&\varepsilon \left( i_{1}+i_{2}+i_{3}+i_{4}\right) .
\end{eqnarray*}%
For the similar situation, we refer the reader to \cite{Taberski2, Riz2}.

Now, using Remark 1 and condition (f) of class $A,$ we get 
\begin{eqnarray*}
i_{1}+i_{2}+i_{3}+i_{4} &=&-\overset{x_{0}-x+\delta }{\underset{x_{0}-x}{%
\int }}\overset{0}{\underset{y_{0}-y-\delta }{\int }}\left \vert K_{\lambda
}\left( t,s\right) \right \vert \left \{ \mu _{1}\left( t-x_{0}+x\right)
\right \} _{t}^{\prime }\left \{ \mu _{2}\left( y_{0}-s-y\right) \right \}
_{s}^{\prime }dsdt \\
&&+\overset{x_{0}-x+\delta }{\underset{x_{0}-x}{\int }}\overset{y_{0}-y}{%
\underset{0}{\int }}\left( \left \vert K_{\lambda }\left( t,s\right)
\right \vert -2K_{\lambda }\left( t,0\right) \right) \left \{ \mu _{1}\left(
t-x_{0}+x\right) \right \} _{t}^{\prime }\left \{ \mu _{2}\left(
y_{0}-s-y\right) \right \} _{s}^{\prime }dsdt.
\end{eqnarray*}%
Hence, the following inequality holds for $I_{121}:$%
\begin{eqnarray*}
I_{121} &\leq &\varepsilon \overset{x_{0}+\delta }{\underset{x_{0}}{\int }}%
\overset{y_{0}}{\underset{y_{0}-\delta }{\int }}\left \vert K_{\lambda
}\left( t-x,s-y\right) \right \vert \left \vert \left \{ \mu _{1}\left(
t-x_{0}\right) \right \} _{t}^{\prime }\right \vert \left \vert \left \{ \mu
_{2}\left( y_{0}-s\right) \right \} _{s}^{\prime }\right \vert dsdt \\
&&+2\varepsilon \left \vert K_{\lambda }\left( 0,0\right) \right \vert \mu
_{1}\left( \delta \right) \mu _{2}\left( \left \vert y_{0}-y\right \vert
\right) .
\end{eqnarray*}%
Analogous computations for $I_{122},$ $I_{123}$ and $I_{124}$ yield:%
\begin{eqnarray*}
I_{122} &\leq &\varepsilon \overset{x_{0}}{\underset{x_{0}-\delta }{\int }}%
\overset{y_{0}}{\underset{y_{0}-\delta }{\int }}\left \vert K_{\lambda
}\left( t-x,s-y\right) \right \vert \left \vert \left \{ \mu _{1}\left(
x_{0}-t\right) \right \} _{t}^{\prime }\right \vert \left \vert \left \{ \mu
_{2}\left( y_{0}-s\right) \right \} _{s}^{\prime }\right \vert dsdt \\
&&+2\varepsilon \left \vert K_{\lambda }\left( 0,0\right) \right \vert \left(
\mu _{1}\left( \delta \right) \mu _{2}\left( \left \vert y_{0}-y\right \vert
\right) +\mu _{2}\left( \delta \right) \mu _{1}\left( \left \vert
x_{0}-x\right \vert \right) \right) +4\varepsilon K_{\lambda }\left( 0\right)
\mu _{1}\left( \left \vert x_{0}-x\right \vert \right) \mu _{2}\left(
\left \vert y_{0}-y\right \vert \right) ,
\end{eqnarray*}%
\begin{eqnarray*}
I_{123} &\leq &\varepsilon \overset{x_{0}}{\underset{x_{0}-\delta }{\int }}%
\overset{y_{0}+\delta }{\underset{y_{0}}{\int }}\left \vert K_{\lambda
}\left( t-x,s-y\right) \right \vert \left \vert \left \{ \mu _{1}\left(
x_{0}-t\right) \right \} _{t}^{\prime }\right \vert \left \vert \left \{ \mu
_{2}\left( s-y_{0}\right) \right \} _{s}^{\prime }\right \vert dsdt \\
&&+2\varepsilon \left \vert K_{\lambda }\left( 0,0\right) \right \vert \mu
_{2}\left( \delta \right) \mu _{1}\left( \left \vert x_{0}-x\right \vert
\right) , \\
I_{124} &\leq &\varepsilon \overset{x_{0}+\delta }{\underset{x_{0}}{\int }}%
\overset{y_{0}+\delta }{\underset{y_{0}}{\int }}\left \vert K_{\lambda
}\left( t-x,s-y\right) \right \vert \left \vert \left \{ \mu _{1}\left(
t-x_{0}\right) \right \} _{t}^{\prime }\right \vert \left \vert \left \{ \mu
_{2}\left( s-y_{0}\right) \right \} _{s}^{\prime }\right \vert dsdt.
\end{eqnarray*}%
Hence the following inequality is obtained for $I_{12}:$%
\begin{eqnarray*}
I_{12} &\leq &\varepsilon \overset{x_{0}+\delta }{\underset{x_{0}-\delta }{%
\int }}\overset{y_{0}+\delta }{\underset{y_{0}-\delta }{\int }}\left \vert
K_{\lambda }\left( t-x,s-y\right) \right \vert \left \vert \left \{ \mu
_{1}\left( \left \vert x_{0}-t\right \vert \right) \right \} _{t}^{\prime
}\right \vert \left \vert \left \{ \mu _{2}\left( \left \vert y_{0}-s\right \vert
\right) \right \} _{s}^{\prime }\right \vert dsdt \\
&&+4\varepsilon \left \vert K_{\lambda }\left( 0,0\right) \right \vert \left(
\mu _{1}\left( \delta \right) \mu _{2}\left( \left \vert y_{0}-y\right \vert
\right) +\mu _{2}\left( \delta \right) \mu _{1}\left( \left \vert
x_{0}-x\right \vert \right) \right)  \\
&&+4\varepsilon \left \vert K_{\lambda }\left( 0,0\right) \right \vert \mu
_{1}\left( \left \vert x_{0}-x\right \vert \right) \mu _{2}\left( \left \vert
y_{0}-y\right \vert \right) 
\end{eqnarray*}%
or equivalently,

\begin{eqnarray*}
I_{12} &\leq &\varepsilon \overset{x_{0}+\delta }{\underset{x_{0}-\delta }{%
\int }}\overset{y_{0}+\delta }{\underset{y_{0}-\delta }{\int }}K_{\lambda
}\left( \sqrt{(t-x)^{2}+(s-y)^{2}}\right) \rho _{1}(\left \vert
t-x_{0}\right \vert )\rho _{2}(\left \vert s-y_{0}\right \vert )dsdt \\
&&+4\varepsilon \left \vert K_{\lambda }\left( 0,0\right) \right \vert
\left( \mu _{1}\left( \delta \right) \mu _{2}\left( \left \vert
y_{0}-y\right \vert \right) +\mu _{2}\left( \delta \right) \mu _{1}\left(
\left \vert x_{0}-x\right \vert \right) \right) \\
&&+4\varepsilon \left \vert K_{\lambda }\left( 0,0\right) \right \vert \mu
_{1}\left( \left \vert x_{0}-x\right \vert \right) \mu _{2}\left( \left
\vert y_{0}-y\right \vert \right) .
\end{eqnarray*}%
The remaining part of the proof is obvious by the hyphotheses (4.1) and
(4.2). Thus the proof is completed.
\end{proof}

The following theorem gives a pointwise approximation of the integral
operators of type (1.3) to the function $f$ at $\mu $-generalized Lebesgue
point of $f\in L_{1}(%
\mathbb{R}
^{2})$.

\begin{theorem}
Suppose that the hypotheses of Theorem 1 are satisfied for $D=%
\mathbb{R}
^{2}$. If $\left( x_{0},y_{0}\right) $\ is a $\mu $-generalized Lebesgue
point of $f\in L_{1}(%
\mathbb{R}
^{2}),$ then \textit{\ }%
\begin{equation*}
\underset{\left( x,y,\lambda \right) \rightarrow \left( x_{0},y_{0},\lambda
_{0}\right) }{\lim }L_{\lambda }\left( f;x,y\right) =f\left(
x_{0},y_{0}\right) .
\end{equation*}%
\textbf{\ }
\end{theorem}

\begin{proof}
The proof of this theorem is quite similar to proof of preceding one and
thus, it is omitted.
\end{proof}

\section{\textbf{Rate of Convergence}}

In this section, we give a theorem concerning the rate of pointwise
convergence of the operators of type (1.3).

\begin{theorem}
Suppose that the hypotheses of Theorem 1 (or Theorem 2) are satisfied.
\end{theorem}

Let 
\begin{equation*}
\Delta (\lambda ,\delta ,x,y)=\overset{x_{0}+\delta }{\underset{x_{0}-\delta 
}{\int }}\overset{y_{0}+\delta }{\underset{y_{0}-\delta }{\int }}\left \vert
K_{\lambda }\left( t-x,s-y\right) \right \vert \rho _{1}(\left \vert
t-x_{0}\right \vert )\rho _{2}(\left \vert s-y_{0}\right \vert )dsdt
\end{equation*}%
for $0<\delta <\delta _{0},$ and the following assumptions are satisfied:$%
\bigskip $

\begin{description}
\item[i] $\Delta (\lambda ,\delta ,x,y)\rightarrow 0$ as $(x,y,\lambda
)\rightarrow (x_{0},y_{0},\lambda _{0})$ for some $\delta >0.$

\item[ii] For every $\gamma >0,$%
\begin{equation*}
\left \vert K_{\lambda }(\gamma ,0)\right \vert =\left \vert K_{\lambda
}(0,\gamma )\right \vert =o(\Delta (\lambda ,\delta ,x,y))
\end{equation*}%
as $(x,y,\lambda )\rightarrow (x_{0},y_{0},\lambda _{0}).$

\item[iii] For every $\gamma >0,$%
\begin{equation*}
\underset{\lambda \rightarrow \lambda _{0}}{\lim }\underset{%
\mathbb{R}
^{2}\backslash \left \langle -\gamma ,\gamma \right \rangle \times
\left \langle -\gamma ,\gamma \right \rangle }{\diint }\left \vert K_{\lambda
}\left( t,s\right) \right \vert dsdt=o(\Delta (\lambda ,\delta ,x,y))
\end{equation*}%
as $(x,y,\lambda )\rightarrow (x_{0},y_{0},\lambda _{0}).$

\item[iv] As $(x,y,\lambda )\rightarrow (x_{0},y_{0},\lambda _{0}),$ we have%
\begin{equation*}
\left \vert \underset{\mathbb{R}^{2}}{\diint }K_{\lambda }\left(
t-x,s-y\right) dsdt-1\right \vert =o(\Delta (\lambda ,\delta ,x,y)).
\end{equation*}
\end{description}

Then, at each $\mu $-generalized Lebesgue point of $f\in L_{1}(D)$ we have%
\begin{equation*}
\left \vert L_{\lambda }\left( f;x,y\right) -f\left( x_{0},y_{0}\right)
\right \vert =o(\Delta (\lambda ,\delta ,x,y))
\end{equation*}

as $(x,y,\lambda )\rightarrow (x_{0},y_{0},\lambda _{0}).$

\begin{proof}
By the hypotheses of Theorem 1, we can write for $\delta >0$%
\begin{eqnarray*}
\left \vert L_{\lambda }\left( f;x,y\right) -f\left( x_{0},y_{0}\right)
\right \vert  &\leq &\left \vert f\left( x_{0},y_{0}\right) \right \vert
\left \vert \underset{\mathbb{R}^{2}}{\diint }K_{\lambda }\left(
t-x,s-y\right) dsdt-1\right \vert  \\
&&+\left \vert f\left( x_{0},y_{0}\right) \right \vert \underset{%
\mathbb{R}
^{2}\backslash \left \langle -\delta ,+\delta \right \rangle \times
\left \langle -\delta ,+\delta \right \rangle }{\diint }\left \vert K_{\lambda
}\left( t,s\right) \right \vert dsdt \\
&&+\underset{(t,s)\in D\backslash D_{22}}{\sup }\left \vert K_{\lambda
}\left( t-x,s-y\right) \right \vert \left( \left \Vert f\right \Vert
_{L_{1}\left( D\right) }+\left \vert f\left( x_{0},y_{0}\right) \right \vert
\left \vert b-a\right \vert \left \vert d-c\right \vert \right) 
\end{eqnarray*}%
\begin{eqnarray*}
&&+\varepsilon \overset{x_{0}+\delta }{\underset{x_{0}-\delta }{\int }}%
\overset{y_{0}+\delta }{\underset{y_{0}-\delta }{\int }}K_{\lambda }\left( 
\sqrt{(t-x)^{2}+(s-y)^{2}}\right) \rho _{1}(\left \vert t-x_{0}\right \vert
)\rho _{2}(\left \vert s-y_{0}\right \vert )dsdt \\
&&+4\varepsilon \left \vert K_{\lambda }\left( 0,0\right) \right \vert \left(
\mu _{1}\left( \delta \right) \mu _{2}\left( \left \vert y_{0}-y\right \vert
\right) +\mu _{2}\left( \delta \right) \mu _{1}\left( \left \vert
x_{0}-x\right \vert \right) \right)  \\
&&+4\varepsilon \left \vert K_{\lambda }\left( 0,0\right) \right \vert \mu
_{1}\left( \left \vert x_{0}-x\right \vert \right) \mu _{2}\left( \left \vert
y_{0}-y\right \vert \right) .
\end{eqnarray*}%
From (i)-(iv) and using class $A$ conditions, we have\ the desired result
for the case $D$ is bounded, that is%
\begin{equation*}
\left \vert L_{\lambda }\left( f;x,y\right) -f\left( x_{0},y_{0}\right)
\right \vert =o(\Delta (\lambda ,\delta ,x,y)).
\end{equation*}%
One may prove the assertion for the case $D=%
\mathbb{R}
^{2}$ following the same steps. Thus, the proof is completed.
\end{proof}

\begin{example}
In this example, we used two dimensional counterpart of the kernel function
used in \cite{karsli}.
\end{example}

Let $\Lambda =[1,\infty ),$ $\lambda _{0}=\infty $ and 
\begin{equation*}
K_{\lambda }(t,s)=\left \{ 
\begin{array}{c}
\lambda ^{2},\text{ \  \ }(t,s)\in \lbrack 0,1/\lambda ]\times \lbrack
0,1/\lambda ], \\ 
0,\text{ \  \  \  \  \  \  \  \  \  \  \  \  \  \  \ }%
\mathbb{R}
^{2}\backslash \lbrack 0,1/\lambda ]\times \lbrack 0,1/\lambda ].%
\end{array}%
\right.
\end{equation*}

It is easy to check that given $K_{\lambda }(t,s)$ is from the class $A.$
Let $\mu _{1}(t)=t,$ $\mu _{2}(s)=s$.

Hence, we obtain

\begin{eqnarray*}
\Delta (\lambda ,\delta ,x,y) &=&\overset{x_{0}+\delta }{\underset{%
x_{0}-\delta }{\int }}\overset{y_{0}+\delta }{\underset{y_{0}-\delta }{\int }%
}\left \vert K_{\lambda }\left( t-x,s-y\right) \right \vert \left \vert \left \{
\mu _{1}(\left \vert t-x_{0}\right \vert )\right \} _{t}^{^{\prime
}}\right \vert \left \vert \left \{ \mu _{2}(\left \vert s-y_{0}\right \vert
)\right \} _{s}^{^{\prime }}\right \vert dsdt \\
&=&4\delta ^{2}\lambda ^{2}.
\end{eqnarray*}%
In order to find for which $\delta >0$ the condition $(i)$ in Theorem 1 is
satisfied, let $\Delta (\lambda ,\delta ,x,y)\rightarrow 0$ as $(x,y,\lambda
)\rightarrow (x_{0},y_{0},\infty )$. Hence 
\begin{equation*}
\underset{(x,y,\lambda )\rightarrow (0,0,0)}{\lim }\Delta (\lambda ,\delta
,x,y)=0
\end{equation*}%
if and only if $\delta =o(1/\lambda ).$ Consequently, the following equation%
\begin{equation*}
\Delta (\lambda ,\delta ,x,y)=\underset{}{\underset{[0,1/\lambda ^{1+\alpha
}]\times \lbrack 0,1/\lambda ^{1+\alpha }]}{\diint }}\lambda ^{2}dsdt
\end{equation*}%
holds for $\alpha \in (0,\infty ).$ From above equation we see that 
\begin{equation*}
\Delta (\lambda ,\delta ,x,y)=O(1/\lambda ^{1+\alpha }).
\end{equation*}

By definition of $K_{\lambda }\left( t,s\right) ,$the conditions (ii) and
(iii) of Theorem 5.1 are satisfied.

The last terms which must be considered are $K_{\lambda }\left( 0,0\right)
\mu _{2}(\left \vert y-y_{0}\right \vert )$ and $K_{\lambda }\left(
0,0\right) \mu _{1}(\left \vert x-x_{0}\right \vert ).$ Analyzing the
following limits, we see that 
\begin{eqnarray*}
\underset{(x,y,\lambda )\rightarrow (0,0,0)}{\lim }K_{\lambda }\left(
0,0\right) \mu _{1}\left \vert x-x_{0}\right \vert &=&\lambda ^{2}\left
\vert x-x_{0}\right \vert =c<\infty , \\
\underset{(x,y,\lambda )\rightarrow (0,0,0)}{\lim }K_{\lambda }\left(
0,0\right) \mu _{2}\left \vert y-y_{0}\right \vert &=&\lambda ^{2}\left
\vert y-y_{0}\right \vert =c^{\prime }<\infty ,
\end{eqnarray*}%
if and only if the rates of convergence of $\lambda ^{2}$ $\rightarrow
\infty $ and $x\rightarrow x_{0},$ $\lambda ^{2}$ $\rightarrow \infty $ and $%
y\rightarrow y_{0}$ are equivalent.

Hence%
\begin{equation*}
\left \vert L_{\lambda }\left( f;x,y\right) -f\left( x_{0},y_{0}\right)
\right \vert =o\left( \Delta (\lambda ,\delta ,x,y)\right) =o\left( 1/\lambda
^{1+\alpha }\right) .
\end{equation*}

\end{document}